\documentclass[12pt]{amsart}
\usepackage{amssymb, enumerate, mdwlist}

\evensidemargin\paperwidth
\advance\evensidemargin-\textwidth
\advance\oddsidemargin-1in
\evensidemargin\oddsidemargin

\topmargin\paperheight
\advance\topmargin-\textheight
\advance\topmargin-1in

\theoremstyle{plain}

\theoremstyle{plain}
\newtheorem{theorem}{Theorem}[section]
\theoremstyle{plain}
\newtheorem{proposition}[theorem]{Proposition}
\theoremstyle{plain}
\newtheorem{lemma}[theorem]{Lemma}
\theoremstyle{plain}

\theoremstyle{plain}

\theoremstyle{plain}

\theoremstyle{plain}

\theoremstyle{definition}
\newtheorem{definition}[theorem]{Definition}
\theoremstyle{remark}
\newtheorem{remark}[theorem]{Remark}
\theoremstyle{remark}

\theoremstyle{remark}

\title{An alternative proof of Kazhdan property for elementary groups}
\author{Masato Mimura}
\address{Mathematical Institute, Tohoku University / EPF Lausanne}
\begin{document}

\begin{abstract}
In 2010, Invent.\ Math., Ershov and Jaikin-Zapirain proved Kazhdan's property $(\mathrm{T})$ for elementary groups. This expository article focuses on presenting an alternative simpler  proof of that. Unlike the original one, our proof supplies \textit{no} estimate of Kazhdan constants. It may be regarded as a specific example of the results in the paper \textit{``Upgrading fixed points without bounded generation''} (arXiv:1505.06728, forthcoming version) by the author.
\end{abstract}


\maketitle
\thispagestyle{empty}

\section{Introduction}
Throughout this article, $\mathcal{H}$ stands for an arbitrary Hilbert space (we do not fix a single Hilbert space: it can be non-separable after taking metric ultraproducts). 
\begin{definition}
For a countable (discrete) group $G$ and $G\geqslant M$, we say that $G\geqslant M$ has \textit{relative property $(\mathrm{FH})$} if for all (affine) isometric $G$-actions $\alpha\colon G\curvearrowright \mathcal{H}$ (for every $\mathcal{H}$), the $M$-fixed point set $\mathcal{H}^{\alpha(M)}$ is non-empty. We say that $G$ has \textit{property $(\mathrm{FH})$} if $G\geqslant G$ has relative property $(\mathrm{FH})$. 
\end{definition}

The Delorme--Guichardet Theorem \cite[Theorem~2.12.4]{BHV} states that (relative) property $(\mathrm{FH})$ is equivalent to (relative) \textit{property $(\mathrm{T})$ of Kazhdan}. Therefore, throughout this article, we use the terminology ``property $(\mathrm{T})$'' for property $(\mathrm{FH})$. See \cite{BHV} for details on these properties. A fundamental example of groups with property $(\mathrm{T})$ is $\mathrm{SL}(n,\mathbb{Z})$ for $n\geq3$ (see \cite[Example~1.7.4.(i)]{BHV}). 

The goal of this article is to provide an alternative proof of the following theorem.

\begin{theorem}[Ershov and Jaikin-Zapirain, Theorem~$1$ in \cite{EJ}]\label{theorem=EJ}
For a finitely generated and  associative ring $R$ with unit and for $n\geq 3$, the elementary group $\mathrm{E}(n,R)$ has property $(\mathrm{T})$.
\end{theorem}

Here, for such $R$ and $n$, the \textit{elementary group} $\mathrm{E}(n,R)$ is the subgroup of $\mathrm{GL}(n,R)$ generated by \textit{elementary matrices} $\{e_{i,j}^r:i\ne j \in \{1,2,\ldots, n\},r\in R\}$. The $e_{i,j}^r$ is defined by $(e_{i,j}^r)_{k,l}:=\delta_{k,l}+r\delta_{i,k}\delta_{j,l}$, where $\delta_{\cdot,\cdot}$ is the Kronecker delta. This theorem greatly generalizes the aforementioned example because for $R=\mathbb{Z}$, Gaussian elimination implies that $\mathrm{SL}(n,\mathbb{Z})=\mathrm{E}(n,\mathbb{Z})$ for all $n\geq 2$. The commutator relation
\[
[e_{i,j}^{r_1},e_{j,k}^{r_2}]=e_{i,k}^{r_1r_2} \quad \textrm{for }i\ne j\ne k\ne i\textrm{ and for $r_1,r_2\in R$}\tag{$\#$}
\]
implies finite generation of $\mathrm{E}(n,R)$ as in Theorem~\ref{theorem=EJ}. Here, $[\gamma_1,\gamma_2]:=\gamma_1\gamma_2\gamma_1^{-1}\gamma_2^{-1}$.

\noindent
\textit{Note.} In \cite{EJ}, $\mathrm{E}(n,R)$ is written as $EL_n(R)$. The ring $R$ may be non-commutative.

For motivations of this result, see Introduction of \cite{EJ}.

\section{Strategies: common points and difference}
Both of the original proof and the new proof in this article consist of following two steps. For $G=\mathrm{E}(n,R)$,
\begin{itemize}
  \item (``Building block'') Show relative properties $(\mathrm{T})$ for $G\geqslant M_j$ for certain subgroups $M_1$,\ldots ,$M_l$.
  \item (``Upgrading'') Upgrade them to property $(\mathrm{T})$ for $G$.
\end{itemize}
In the first step, both of us employ the following.
\begin{theorem}[Kassabov,  Corollary~$2.8$ in \cite{kassabov}]\label{theorem=kassabov}For $R$ and $n$ as in Theorem~$\ref{theorem=EJ}$, and for all $i\ne j \in \{1,\ldots ,n\}$, $\mathrm{E}(n,R)\geqslant G_{i,j}$ has relative property $(\mathrm{T})$. Here, $G_{i,j}:=\langle e_{i,j}^r:r\in R\rangle (\simeq (R,+) )$.
\end{theorem}
\noindent
\textit{Note.} Kassabov showed it in terms of the original definition of property $(\mathrm{T})$. As we mentioned above, this is equivalent to our property $(\mathrm{T})$ (property $(\mathrm{FH})$).

In fact, Kassabov's original form is for the pair $\mathrm{E}(n-1,R)\ltimes R^{n-1} \trianglerighteq R^{n-1}$ and ${}^t(R^{n-1})\rtimes \mathrm{E}(n-1,R) \trianglerighteq {}^t(R^{n-1})$ for all $n\geq 3$. Here for the former, $\mathrm{E}(n-1,R)$ acts on $R^{n-1}$ (column vectors) by the left multiplication; for the latter, it acts on ${}^t(R^{n-1})$ (row vectors) by the right multiplication. The proof for these pairs is  by spectral theory associated with  unitary representations of abelian groups; see also \cite{shalom1999} and \cite[Sections~4.2 and 4.3]{BHV}. To deduce Theorem~\ref{theorem=kassabov} from this, embed $\mathrm{E}(n-1,R)\ltimes R^{n-1} $ and ${}^t(R^{n-1})\rtimes \mathrm{E}(n-1,R)$ into $\mathrm{E}(n,R)$ in several ways. Note that if $G\geqslant G_0\geqslant M_0\geqslant M$, then relative property $(\mathrm{T})$ for $G\geqslant M$ follows from that for $G_0\geqslant M_0$.

The difference between the original proof in \cite{EJ} and ours lies in ``Upgrading''.
\begin{enumerate}[$(1)$]
  \item In the original proof, upgrading is \textit{extrinsic}: They consider \textit{angles} between fixed point subspaces, and showed that if angles are sufficiently close to being orthogonal, then upgrading works. 

More precise form is as follows. Since their original argument deals with the original formulation of property $(\mathrm{T})$, we here rather sketch the argument in formulation of Lavy \cite{lavy}. For two (non-empty) affine subspaces $\mathcal{K}_1$, $\mathcal{K}_2$ of $\mathcal{H}$, they define 
\[
\cos \angle (\mathcal{K}_1,\mathcal{K}_2):=\sup_{0\ne\xi \in \mathcal{K}_1'/(\mathcal{K}_1'\cap\mathcal{K}_2'),\ 0\ne \eta \in \mathcal{K}_2'/(\mathcal{K}_1'\cap\mathcal{K}_2')} \frac{|\langle \xi,\eta \rangle |}{\|\xi\|\|\eta\|}\in [0,1],
\]
where $\mathcal{K}_i'$, $i=1,2$, is the \textit{linear} subspace obtained by parallel transformation of $\mathcal{K}_i$. The symbol $\langle \cdot,\cdot \rangle$ means the (induced) inner product on $\mathcal{H}/(\mathcal{K}_1'\cap\mathcal{K}_2')$. They say $\mathcal{K}_1$ and $\mathcal{K}_2$ are \textit{$\epsilon$-orthogonal} if $\cos \angle (\mathcal{K}_1,\mathcal{K}_2)<\epsilon$. 

They showed existence of  (explicit) $(\epsilon_{i,j})_{1\leq i<j\leq l}$ (for every $l\geq 3$) with the following property. For $G\geqslant H_i$, $1\leq i\leq l$, with $\langle H_1,\ldots ,H_l\rangle =G$, if $\mathcal{H}^{\alpha(H_i)}$ and  $\mathcal{H}^{\alpha(H_j)}$ are  $\epsilon_{i,j}$-orthogonal for all $i\ne j$ and for all affine isometric actions $\alpha\colon G\curvearrowright \mathcal{H}$, then ``relative properties $(\mathrm{T})$ for $G\geqslant H_{i,j}:=\langle H_i,H_j\rangle$ for all $i\ne j$'' imply property $(\mathrm{T})$ for $G$. See \cite[Subsection~1.2 and Section~2]{lavy} as well as \cite{EJ}.

\noindent
\textit{Note.} To apply this ``(extrinsic) upgrading'' criterion for $G=\mathrm{E}(n,R)$ for a general $R$ (say, $R=\mathbb{Z}\langle x,y\rangle$, a non-commutative polynomial ring), quite delicate estimate of spectral quantities ($\epsilon_{i,j}$ for $\epsilon_{i,j}$-orthogonality) is needed. This, together with the proof of the criterion, makes their proof heavier. In return, they obtain estimation of \textit{Kazhdan constants} (see \cite[Remark~1.1.4]{BHV}).
  \item Our upgrading is \textit{intrinsic}: Our criterion is stated only in terms of group structure, and not of group actions. This makes our proof rather simpler than the original one. The price to pay, however, is that we do \textit{not} obtain any estimate for Kazhdan constants. This is a special case of our \textit{$($intrinsic$)$ upgrading without bounded generation} \cite{mimurass}. Our upgrading result is as follows.
\end{enumerate}

\begin{theorem}[\textit{Upgrading} theorem in this article]\label{theorem=synthesis}
Let $G$ be a finitely generated group and $M,L\leqslant G$ satisfy $\langle M,L\rangle=G$. Assume the following hypothesis.

\noindent
\underline{$(\mathrm{GAME}^{\mathrm{inn}})$ hypothesis}: The player can win the $(\mathrm{Game}^{\mathrm{inn}})$ for $(G,M,L)$, which is defined in Section~$\ref{section=game}$.

Then, relative properties $(\mathrm{T})$ for $G\geqslant M$ and $G\geqslant L$ imply property $(\mathrm{T})$ for $G$.
\end{theorem}

\section{$(\mathrm{Game}^{\mathrm{inn}})$}\label{section=game}

Let $G$ be a group, and $M,L\leqslant G$ satisfy $\langle M,L\rangle =G$. The ``\textit{$(\mathrm{Game}^{\mathrm{inn}})$ for $(G,M,L)$}'' is a one-player game. Here, we fix $(G,M,L)$ and keep them unchanged. For each ordinal $\iota$, we set $H_1^{(\iota)}\leqslant G$ and $H_2^{(\iota)}\leqslant G$. Those two subgroups $H_1$ and $H_2$, indexed by $\iota$, are respectively ``upgraded'' (enlarged). The rules are the following.

\begin{itemize}
  \item For the initial stage ($\iota=0$), $H_1^{(0)}=M$ and $H_2^{(0)}=L$.
  \item The player wins if there exists $\iota$ for which the player can set as either $H_1^{(\iota)}=G$ or $H_2^{(\iota)}=G$.
  \item For each non-zero limit ordinal $\iota$, the player can upgrade $H_1$ and $H_2$ as 
\renewcommand{\arraystretch}{1.5}
\begin{table}[h]
 \begin{tabular}{c|c}
   $H_1^{(\iota)}$  & $H_2^{(\iota)}$ \\ \hline 
$\bigcup_{\iota'<\iota} H_1^{(\iota')}$ & $\bigcup_{\iota'<\iota} H_2^{(\iota')}$.
  \end{tabular}
\end{table}
\renewcommand{\arraystretch}{1.0}

Here, this table means that for each $j=1,2$, $H_j^{(\iota)}$ is defined as $\bigcup_{\iota'<\iota} H_j^{(\iota')}$. We use similar tables to indicate this kind of upgrading.
  \item For each non-zero $\iota$ such that $\iota-1$ exists, the player is allowed to take either one of admissible moves for upgrading: type $(\mathrm{I})$ move and type $(\mathrm{I\hspace{-.1em}I}^{\mathrm{inn}})$ move.

\end{itemize}

\noindent
(Rules of the admissible moves: upgrading from $\iota-1$ to $\iota$) 
\begin{itemize}
  \item \textit{Type $(\mathrm{I})$ move}. Pick a subset $(Q^{(\iota)}=)Q\subseteq G$ such that for all $\gamma \in Q$,
\[
\gamma H_1^{(\iota -1)}\gamma^{-1}\geqslant M \quad \mathrm{and}\quad \gamma H_2^{(\iota-1)}\gamma^{-1}\geqslant L.
\]
Then, upgrade $H_1$ and $H_2$ as
\renewcommand{\arraystretch}{1.5}
\begin{table}[h]
 \begin{tabular}{c|c}
   $H_1^{(\iota)}$  & $H_2^{(\iota)}$ \\ \hline 
$\langle H_1^{(\iota-1)},Q\rangle$ & $\langle H_2^{(\iota-1)},Q\rangle$.
  \end{tabular}
\end{table}
\renewcommand{\arraystretch}{1.0}
  \item \textit{Type $(\mathrm{I\hspace{-.1em}I}^{\mathrm{inn}})$ move}. Pick a subset $(W^{(\iota)}=)W \subseteq G$ such that for all $w \in W$,
\[
wH_2^{(\iota-1)}w^{-1}\geqslant M\quad \mathrm{and}\quad wH_1^{(\iota-1)}w^{-1} \geqslant L.
\]
Then, upgrade $H_1$ and $H_2$  as
\renewcommand{\arraystretch}{1.5}
\begin{table}[h]
 \begin{tabular}{c|c}
   $H_1^{(\iota)}$  & $H_2^{(\iota)} $\\ \hline 
$\langle H_1^{(\iota-1)},\bigcup_{w \in W}w^{-1} H_2 ^{(\iota-1)}w\rangle$& $\langle H_2^{(\iota-1)},\bigcup_{w \in W}w^{-1}H_1^{(\iota-1)}w\rangle$.
  \end{tabular}
\end{table}
\renewcommand{\arraystretch}{1.0}
\end{itemize}
Here, $(\mathrm{I\hspace{-.1em}I}^{\mathrm{inn}})$ means the following. In more general ``$(\mathrm{Game})$", there is a notion of type $(\mathrm{I\hspace{-.1em}I}_{(12)})$ moves. Type $(\mathrm{I\hspace{-.1em}I}^{\mathrm{inn}})$ moves are  restricted type $(\mathrm{I\hspace{-.1em}I}_{(12)})$ moves, where we only consider inner conjugations as group automorphisms. Compare with Main Theorems in \cite{mimurass}.

These moves in $(\mathrm{Game}^{\mathrm{inn}})$ represent ways of upgrading. Rough meaning is, for an affine isometric action $\alpha\colon G\curvearrowright \mathcal{H}$, ``if $\xi \in \mathcal{H}^{\alpha(M)}$ and $\eta \in \mathcal{H}^{\alpha(M)}$ are chosen in a special manner, then \textit{information on $\xi$ and $\eta$ is automatically upgraded; more precisely, in each stage $\iota$ of $(\mathrm{Game}^{\mathrm{inn}})$, $\xi\in \mathcal{H}^{\alpha(H_1^{(\iota)})}$ and $\eta \in \mathcal{H}^{\alpha(H_2^{(\iota)})}$}.'' See Proposition~\ref{proposition=si} for the rigorous statement. There might be some formal similarity to \textit{Mautner phenomena}, which is upgrading process in continuous setting (unitary representations of Lie groups) with the aid of  limiting arguments; see \cite[Lemma~1.4.8]{BHV}.

Let us see how we employ this criterion to the case of $G=\mathrm{E}(n,R)$.

\begin{proof}[Proof of ``Theorem~$\ref{theorem=synthesis}$ implies Theorem~$\ref{theorem=EJ}$'']
Let $R$ and $n$ be as in Theorem~\ref{theorem=EJ}. Set $G=\mathrm{E}(n,R)$, $M=
\left(\begin{array}{cc}
I_{n-1} & R^{n-1} \\
0 & 1 
\end{array}
\right)(\simeq (R^{n-1},+))$, and 
$L=
\left(\begin{array}{cc}
I_{n-1} & 0 \\
{}^t(R^{n-1}) & 1 
\end{array}
\right)$. 
Note that by $(\#)$, $\langle M,L\rangle =G$.

The aforementioned original form of Theorem~\ref{theorem=kassabov} by Kassabov implies
\begin{lemma}\label{lemma=kassabov}
These $G\geqslant M$ and $G\geqslant L$ have relative property $(\mathrm{T})$.
\end{lemma}

What remains is to check hypothesis $(\mathrm{GAME}^{\mathrm{inn}})$ for the triple $(G,M,L)$ above. First, take type $(\mathrm{I})$ move with $Q=
\left(\begin{array}{cc}
\mathrm{E}({n-1},R) & 0 \\
0 & 1 
\end{array}
\right)$. 
Then, the upgrading is
\renewcommand{\arraystretch}{1.7}
\begin{table}[h]
 \begin{tabular}{c|c}
   $H_1^{(1)}$  & $H_2^{(1)}$ \\ \hline 
$\langle H_1^{(0)},Q\rangle=\left(\begin{array}{cc}
\mathrm{E}({n-1},R) & R^{n-1} \\
0 & 1 
\end{array}
\right)$&$ \langle H_2^{(0)},Q\rangle=\left(\begin{array}{cc}
\mathrm{E}({n-1},R) & 0 \\
{}^t(R^{n-1}) & 1 
\end{array}
\right)$.
  \end{tabular}
\end{table}
\renewcommand{\arraystretch}{1.0}

Finally, take type $(\mathrm{I\hspace{-.1em}I}^{\mathrm{inn}})$ move with $W=\left\{w\right\}$. Here, $w=\left(\begin{array}{ccc}
0 & 0 & 1\\
0 & I_{n-3,1} & 0 \\
1 & 0 & 0
\end{array}
\right)$, where $I_{n-3,1}$ denotes the diagonal matrix with diagonals $1,\ldots ,1$ ($n-3$ times) and $-1$. Note that $w=(e_{n-1,n}^1e_{n,n-1}^{-1}e_{n-1,n}^1)^2(e_{1,n}^1e_{n,1}^{-1}e_{1,n}^1) \in G$. The \textit{miracle} here is 
\begin{align*}
 wH_2^{(1)}w^{-1}=\left(\begin{array}{cc}
1 & {}^t(R^{n-1}) \\
0 & \mathrm{E}({n-1},R) 
\end{array}
\right) \geqslant M,\ \mathrm{and} \quad 
wH_1^{(1)}w^{-1}=\left(\begin{array}{cc}
1 & 0 \\
R^{n-1} &  \mathrm{E}({n-1},R)
\end{array}
\right) \geqslant L.
\end{align*}
Therefore, our upgrading process goes as follows.

\renewcommand{\arraystretch}{1.7}
\begin{table}[h]
 \begin{tabular}{c|c}
   $H_1^{(2)}$  & $H_2^{(2)}$ \\ \hline 
$\langle H_1^{(1)},w^{-1}H_2^{(1)}w\rangle=\left\langle\left(\begin{array}{cc}
\mathrm{E}({n-1},R) & R^{n-1} \\
0 & 1 
\end{array}
\right),\left(\begin{array}{cc}
1 & {}^t(R^{n-1}) \\
0 & \mathrm{E}(n-1,R) 
\end{array}
\right) \right\rangle$&$ \langle H_2^{(1)},w^{-1}H_1^{(1)}w\rangle$.
  \end{tabular}
\end{table}
\renewcommand{\arraystretch}{1.0}

Since $[e_{n,2}^r,e_{2,1}^1]=e_{n,1}^r$ for $n\geq 3$ by $(\#)$, the new $H_1^{(2)}$ and $H_2^{(2)}$ both equal $G$. 

Therefore, we are done.
\end{proof}
We will prove Theorem~\ref{theorem=synthesis} in Section~\ref{section=proof}. The combination of these proofs provides our new proof of Theorem~\ref{theorem=EJ}.

\begin{remark}\label{remark=bg}
The upgrading argument above is intrinsic, but \textit{it does not use any form of }(\textit{non-trivial}) \textit{bounded generation}, whose definition is as follows.
\begin{definition}\label{definition=bg}
Let $(1_G\in) U=U^{-1}$ and $X$ be non-empty subsets of $G$. We say that $U$ \textit{boundedly generates} $X$ if there exists $N\in \mathbb{N}$ such that $U^N\supseteq X$, where $U^N(\subseteq G)$ denotes the image of $U\times \cdots \times U$ (the $N$-time direct product) by the iterated multiplication map $G\times \cdots \times G \to G;\ (g_1,\ldots ,g_N)\mapsto g_1\cdots g_N$.
\end{definition}
\noindent
\textit{Note.} In some literature, ``bounded generation'' is used in the following very restricted way: $U=\bigcup_{1\leq j\leq l}C_j$ for $C_j$ cyclic subgroups and $X=G$. Our convention is much more general. 

Study of intrinsic upgrading was initiated and developed by works of Shalom \cite{shalom1999}, \cite{shalom2006}. Here we recall the definition of \textit{displacement functions}.

\begin{definition}\label{definition=displacement}
Let $\alpha\colon G\curvearrowright \mathcal{H}$ be an affine isometric action of a countable discrete group. Let $A\subseteq G$ be a non-empty subset. Then, the \textit{displacement} (function) \textit{over $A$} is the function
\[
\mathrm{disp}_{\alpha}^{A}\colon \mathcal{H}\to \mathbb{R}_{\geq 0}\cup\{+\infty\};\quad \mathcal{H}\ni \zeta \mapsto \mathrm{disp}_{\alpha}^{A}(\zeta)=\sup_{a\in A}\|\alpha(a)\cdot \zeta-\zeta\|.
\]
\end{definition}
Note that $\mathrm{disp}_{\alpha}^{A}$ is $2$-Lipschitz if it takes a finite value at some point $\zeta\in \mathcal{H}$ (this condition does not depend on the choice of $\zeta$).

\begin{proposition}[Shalom's bounded generation argument, \cite{shalom1999}]\label{proposition=boundedgeneration}
Let $\alpha\colon G\curvearrowright \mathcal{H}$ be an affine isometric action of a countable discrete group. Let $M_1,\ldots,M_l$ be subgroups of $G$. Assume that for each $1\leq j\leq l$, $\mathcal{H}^{\alpha(M_j)}\ne \emptyset$. Assume the following $\mathrm{bounded\ generation\ hypothesis}$: $\bigcup_{1\leq j\leq l}M_j$ boundedly generates $G$. 

Then, $\mathcal{H}^{\alpha(G)}\ne \emptyset.$
\end{proposition}

\begin{proof}
By triangle inequality and isometry of $\alpha$, observe that for non-empty subsets $A$ and $B$ in $G$ and for every $\zeta\in \mathcal{H}$, $\mathrm{disp}^{AB}_{\alpha}(\zeta)\leq \mathrm{disp}^{A}_{\alpha}(\zeta)+\mathrm{disp}^{B}_{\alpha}(\zeta)$. Here $AB(\subseteq G)$ is defined as the image of $A\times B$ by the multiplication map $(g_1,g_2)\mapsto g_1g_2$. 

Let $U=\bigcup_{1\leq j\leq l}M_j$. Pick one $\zeta\in \mathcal{H}$. By assumption of $\mathcal{H}^{\alpha(M_j)}\ne \emptyset$, $\mathrm{disp}^{U}_{\alpha}(\zeta)<\infty$. By bounded generation hypothesis, it follows that $\mathrm{disp}^{G}_{\alpha}(\zeta)<\infty$. It means that, the $G$-orbit of $\zeta$, $\alpha(G)\cdot \zeta$, is bounded. Then, the (unique) Chebyshev center \cite[Lemma~2.2.7]{BHV} of that set is $\alpha(G)$-fixed.
\end{proof}

In Shalom's intrinsic upgrading, some forms of (non-trivial) bounded generation hypotheses were essential. This \textit{was} a bottle-neck for intrinsic upgrading, because these hypotheses are super-strong in general. For instance, bounded generation for $\mathrm{E}(n,R)$, $n\geq 3$, by elementary matrices ($\bigcup_{i\ne j}G_{i,j}$, where $G_{i,j}$ is as in Theorem~\ref{theorem=kassabov}) is true for $R=\mathbb{Z}$ (Carter--Keller; see \cite[Section~4.1]{BHV}), \textit{false} for $R=\mathbb{C}[t]$ for every $n(\geq 2)$ (van der Kallen), and \textit{open} in many cases, even for $R=\mathbb{Z}[t]$.
\end{remark}

\section{Proof of Theorem~\ref{theorem=synthesis}}\label{section=proof}
\subsection{The upshot of upgrading}\label{subsection=si}
The idea of the proof is inspired by Shalom's second intrinsic upgrading \cite[4.III]{shalom2006} (with bounded generation); Shalom himself called it \textit{algebraization}. The \textit{upshot} of our upgrading is the following.

\begin{proposition}[The \textit{upshot} of our upgrading]\label{proposition=si}
Let $G$ be a finitely generated group.  Let $M,L\leqslant G$ with $\langle M,L\rangle =G$. Let $\alpha\colon G\curvearrowright \mathcal{H}$ be an affine isometric action. Assume that the linear part $\pi$ of $\alpha$ does not have non-zero $G$-invariant vectors. Assume $\mathcal{H}^{\alpha(M)}\ne \emptyset$ and $\mathcal{H}^{\alpha(L)}\ne \emptyset$. Assume, besides, that $(\xi,\eta)\in \mathcal{H}^{\alpha(M)}\times \mathcal{H}^{\alpha(L)}$ realizes the distance 
\[
D:=\mathrm{dist}(\mathcal{H}^{\alpha(M)},\mathcal{H}^{\alpha(L)})(=\inf \left\{\|\zeta_1-\zeta_2\|:\zeta_1\in \mathcal{H}^{\alpha(M)},\, \zeta_2\in \mathcal{H}^{\alpha(L)} \right\}).
\]

Then, in each stage $\iota$ in $(\mathrm{Game}^{\mathrm{inn}})$ for $(G,M,L)$, $\xi \in \mathcal{H}^{\alpha(H_1^{(\iota)})}$ and $\eta \in \mathcal{H}^{\alpha(H_2^{(\iota)}))}$.
\end{proposition}
Recall an affine isometric action $\alpha$ is decomposed into linear part $\pi$ (unitary representation) and cocycle part $b$: $\alpha(g)\cdot \zeta=\pi(g)\zeta+b(g)$ (see \cite[Section~2.1]{BHV}).

The key to the proof of Proposition~$\ref{proposition=si}$ is the following observation due to Shalom.

\begin{lemma}[\textit{Shalom's parallelogram argument}; see 4.III.6 in \cite{shalom2006}]\label{lemma=shalom}
In the setting as in Proposition~$\ref{proposition=si}$, the realizer  of $D$ is unique.
\end{lemma}
\begin{proof}
Let $(\xi',\eta')$ be another realizer. Take midpoints $(m_1,m_2)$, where $m_1=(\xi+\xi')/2$ and $m_2=(\eta+\eta')/2$. Observe that $(m_1,m_2)$ is again a realizer of $D$; indeed, $m_1\in \mathcal{H}^{\alpha(M)}$, $m_2\in \mathcal{H}^{\alpha(L)}$, and $\|m_1-m_2\|\leq D$ (by triangle inequality).

Then, by strict convexity of $\mathcal{H}$, $\xi-\eta=\xi'-\eta'$. Note that $\xi-\xi'\in \mathcal{H}^{\pi(M)}$ and $\eta-\eta'\in \mathcal{H}^{\pi(L)}$. Therefore, $\xi-\xi'=\eta-\eta'\in \mathcal{H}^{\pi(G)}$; hence these equal $0$.
\end{proof}
\begin{proof}[Proof of Proposition~$\ref{proposition=si}$]
By (transcendental) induction on $\iota$. For $\iota=0$, the assertion holds. We proceed in induction step. If $\iota$ is a non-zero limit ordinal, then the upgrading from all $\iota'<\iota$ to $\iota$ in the rules of $(\mathrm{Game}^{\mathrm{inn}})$ straightforwardly works. Finally, we deal with the case where $\iota-1$ exists. Assume that $\xi \in \mathcal{H}^{\alpha(H_1^{(\iota-1)})}$ and $\eta \in \mathcal{H}^{\alpha(H_2^{(\iota-1)})}$; we take a new move.

\begin{itemize}
  \item \textit{Case~$1$. New move is of type $(\mathrm{I})$}. This case was essentially done by Shalom \cite[4.III.6]{shalom2006}. Let $\gamma\in Q$. The conditions $\gamma H_1^{(\iota-1)}\gamma^{-1}\geqslant M$ and $\gamma H_2^{(\iota-1)} \gamma^{-1}\geqslant L$ are imposed on $Q$ exactly in order to ensure that $\alpha(\gamma)\cdot \xi \in \mathcal{H}^{\alpha(M)}$ and that $\alpha(\gamma)\cdot \eta \in \mathcal{H}^{\alpha(L)}$. By isometry of $\alpha$, $(\xi,\eta)$ and $(\alpha(h)\cdot \xi,\alpha(h)\cdot \eta)$ are two realizers of $D$, and Lemma~\ref{lemma=shalom} applies. Therefore, we obtain that $\xi \in \mathcal{H}^{\alpha(Q)}\cap \mathcal{H}^{\alpha(H_1^{(\iota-1)})}=\mathcal{H}^{\alpha(\langle H_1^{(\iota-1)},Q\rangle)}$; similarly, $\eta\in \mathcal{H}^{\alpha(\langle H_2^{(\iota-1)},Q\rangle)}$.
  \item \textit{Case~$2$. New move is of type $(\mathrm{I\hspace{-.1em}I}^{\mathrm{inn}})$}. Let $w\in W$. Then, the condition on $w$ implies that $\alpha(w)\cdot \eta \in \mathcal{H}^{\alpha(M)}$ and that $\alpha(w)\cdot \xi \in \mathcal{H}^{\alpha(L)}$. Hence, this time $(\alpha(w)\cdot \eta, \alpha(w)\cdot \xi)$ is another realizer of $D$. Again by Lemma~\ref{lemma=shalom}, $\alpha(w)\cdot \xi=\eta$. By recalling $\eta \in \mathcal{H}^{\alpha(H_2^{(\iota-1)})}$, we conclude that $(\mathcal{H}^{\alpha(H_1^{(\iota-1)})}\ni )\xi \in  \mathcal{H}^{\alpha(w^{-1}H_2^{(\iota-1)}w)}$ for all $w\in W$. Similarly, $(\mathcal{H}^{\alpha(H_2^{(\iota-1)})}\ni)\eta\in \mathcal{H}^{\alpha(w^{-1}H_1^{(\iota-1)}w)}$ for all $w\in W$.
\end{itemize}
\end{proof}

\subsection{Metric ultraproducts, scaling limits, and realizers of the distance}\label{subsection=up}
Proposition~\ref{proposition=si} might look convincing, but there is a \textit{gap} to conclude Theorem~\ref{theorem=synthesis}. In general, there is \textit{no} guarantee on the existence of \textit{realizers} $(\xi, \eta)$ of $D$. 

This gap will be fixed by well-known Propositions~\ref{proposition=gs} and \ref{proposition=shalom} below. Nevertheless,   we include (sketchy) proofs for the reader's convenience. They employ metric ultraproducts. The reader who is familiar with this topic may skip this subsection.

\begin{definition}[uniform action]\label{definition=uniform}
Let $G$ be a finitely generated group and let $S=S^{-1}$ be a  finite   generating set of $G$. 
Let $\alpha\colon G\curvearrowright \mathcal{H}$ be an affine isometric action. The action $\alpha$ is said to be \textit{$(S,1)$-uniform} if $\inf_{\zeta\in \mathcal{H}}\mathrm{disp}_{\alpha}^S(\zeta)\geq 1$. For a fixed $S$, we simply say $\alpha$ to be $1$\textit{-uniform} for short.
\end{definition}

Fix a finitely generated group $G$; fix moreover a finite generating set $S$.  We set the following three classes of actions and Hilbert spaces.
\begin{itemize}
  \item $\mathcal{C}:=\{(\alpha,\mathcal{H})\}$, where $\mathcal{H}$ is a Hilbert space and $\alpha\colon G\curvearrowright \mathcal{H}$ is an affine isometric action.
  \item $\mathcal{C}^{\mathrm{non}\textrm{-}\mathrm{fixed}}:=\{(\alpha,\mathcal{H}):(\alpha,\mathcal{H})\in \mathcal{C},\ \mathcal{H}^{\alpha(G)}=\emptyset\}$.
  \item $\mathcal{C}^{1\textrm{-}\mathrm{uniform}}:=\{(\alpha,\mathcal{H}):(\alpha,\mathcal{H} )\in \mathcal{C},\ \textrm{$\alpha$ is $1$-uniform.}\}$.
\end{itemize}
Then, $\mathcal{C}^{1\textrm{-}\mathrm{uniform}}$ is a subclass of $\mathcal{C}^{\mathrm{non}\textrm{-}\mathrm{fixed}}$. The failure of property $(\mathrm{T})$ for $G$ exactly says that $\mathcal{C}^{\mathrm{non}\textrm{-}\mathrm{fixed}}\ne \emptyset$.

In the two propositions below, let $(G,S)$ be as in Definition~\ref{definition=uniform}.

\begin{proposition}[A special case of the \textit{Gromov--Schoen argument}; see also 4.III.2 in \cite{shalom2006}]\label{proposition=gs}
Assume $G$ fails to have property $(\mathrm{T})$. Then, $\mathcal{C}^{1\textrm{-}\mathrm{uniform}}\ne \emptyset$.
\end{proposition}

\begin{proposition}[Shalom, 4.III.3--4 in \cite{shalom2006}]\label{proposition=shalom}
Let $M\leqslant G$ and $L\leqslant G$ be subgroups with $\langle M,L\rangle =G$. Assume that $M\leqslant G$ and $L\leqslant G$ have relative property $(\mathrm{T})$, and that $\mathcal{C}^{1\textrm{-}\mathrm{uniform}}\ne \emptyset$. 

Then, $D:=\inf \{\|\xi -\eta\|: (\alpha,\mathcal{H}) \in \mathcal{C}^{1\textrm{-}\mathrm{uniform}},\ \xi\in \mathcal{H}^{\alpha(M)},\ \eta\in \mathcal{H}^{\alpha(L)}\}$ is realized.
\end{proposition}

Here we briefly recall the definitions on (pointed) metric ultraproducts. See a survey \cite{stalder} for more details. \textit{Ultrafilters} $\mathcal{U}$ on $\mathbb{N}$ have one-to-one correspondence to $\{0,1\}$-valued probability \textit{mean}s (that means, \textit{finitely additive} measures $\mu$, such that $\mu(\mathbb{N})=1$, defined over all subsets in $\mathbb{N}$). The correspondence is in the following manner: $\mathcal{U}=\{A\subseteq \mathbb{N}: \mu(A)=1\}$. A \textit{principal} ultrafilter corresponds to the Dirac mass at a point in $\mathbb{N}$. \textit{Non-principal} ultrafilters correspond to  all the other ones. 

In what follows, we fix a non-principal ultrafilter $\mathcal{U}$. For real numbers $r_n$, we write as $\lim_{\mathcal{U}}r_n=r_{\infty}$, if for all $\epsilon >0$, $\{n\in \mathbb{N}: |r_{\infty}-r_n|<\epsilon\}\in \mathcal{U}$. Then, it is well-known that  every \textit{bounded} real sequence $(r_n)_n$ has a (unique) limit with respect to $\mathcal{U}$. Since $\mathcal{U}$ is non-principal, if $\lim_{n\to \infty} r_n$ exists, then it coincides with $\lim_{\mathcal{U}}r_n$.

Let $((X_n,d_n,z_n))_{n\in \mathbb{N}}$ be a sequence of pointed metric spaces. Let $\ell_{\infty}$-$\prod_n (X_n,z_n):=\{(x_n)_n: x_n\in X_n,\ \sup_{n\in \mathbb{N}} d_n(x_n,z_n)<\infty\}$, and $d_{\infty}((x_n)_n,(y_n)_n):=\lim_{\mathcal{U}}d_n(x_n,y_n)$.  Finally, define the \textit{pointed metric ultraproduct} $(X_{\mathcal{U}},d_{\mathcal{U}},z_{\mathcal{U}})$ as follows.
\[
X_{\mathcal{U}}:=\ell_{\infty}\textrm{-}\prod_n (X_n,z_n)/\sim_{d_{\infty}=0}, 
\]
$d_{\mathcal{U}}$ is the induced (genuine) metric, and $z_{\mathcal{U}}:=[(z_n)_n]$. Here $\sim_{d_{\infty}=0}$ denotes that we identify all of two sequences in $\ell_{\infty}$-$\prod_n (X_n,z_n)$ whose distance in $d_{\infty}$ is zero, and $[\cdot]$ means the equivalence class in $\sim_{d_{\infty}=0}$. This is also written as $\lim_{\mathcal{U}}(X_n,d_n,z_n)$. We can show that metric ultraproducts of (affine) Hilbert spaces are again  (affine) Hilbert spaces (because Hilbert spaces are characterized in terms of inner products).

We  fix $(G,S)$. For a sequence of pointed (isometric) $G$-actions $(\alpha_n,(X_n,d_n),z_n)$ \textit{that satisfies}
\[
\sup_{n}\mathrm{disp}^S_{\alpha_n}(z_n) <\infty,\tag{$\diamond$}
\]
we can define the \textit{pointed metric ultraproduct action} $\alpha_{\mathcal{U}}$ on $(X_{\mathcal{U}},d_{\mathcal{U}},z_{\mathcal{U}})$ by
\[
\alpha_{\mathcal{U}}(g)\cdot [(x_n)_n]:=[(\alpha_n(g)\cdot x_n)_n]. 
\]
This is also written as $\lim_{\mathcal{U}}(\alpha_n,X_n,z_n)$.

\begin{proof}[Proof of Proposition~$\ref{proposition=gs}$]
Let $(\alpha,\mathcal{H})\in \mathcal{C}^{\mathrm{non}\textrm{-}\mathrm{fixed}}(\ne \emptyset)$. By completeness of $\mathcal{H}$, we can find a sequence $(\zeta_n)_{n\in \mathbb{N}}$ with the following property. For all $\chi \in \mathcal{H}$ with $\|\chi -\zeta_n\| \leq (n+1)\mathrm{disp}^S_{\alpha}(\zeta_n)$, it holds that $\mathrm{disp}^S_{\alpha}(\chi)\geq \mathrm{disp}^S_{\alpha}(\zeta_n)/2(>0)$; see \cite[Lemma~3.3]{stalder}. 

Then, the ultraproduct $\lim_{\mathcal{U}}(\alpha, (\mathcal{H}, r_n\|\cdot \|), \zeta_n)$ is well-defined and $1$-uniform. Here $r_n:=2 (\mathrm{disp}^S_{\alpha}(\zeta_n))^{-1}$.
\end{proof}

\begin{proof}[Proof of Proposition~$\ref{proposition=shalom}$]
Observe that this infimum is over a non-empty set. Let $((\alpha_n,\mathcal{H}_n,\xi_n,\eta_n))_n$ be a sequence that ``asymptotically realizes" $D$ as $n\to \infty$. More precisely, assume that $\|\xi_n -\eta_n\|\leq D+2^{-n}$. 
We claim that $((\alpha_n,\mathcal{H}_n,\xi_n))_{n\in \mathbb{N}}$ satisfies $(\diamond)$. Indeed, note that $\mathrm{disp}^{M}_{\alpha_n}(\xi_n)=0$ and $\mathrm{disp}^{L}_{\alpha_n}(\eta_n)=0$. The latter implies that $\mathrm{disp}^{L}_{\alpha_n}(\xi_n)\leq 2(D+1)$. Observe that there exists $N\in \mathbb{N}$ such that $(M\cup L)^N \supseteq S$, because $\langle M,L\rangle=G$ and $|S|<\infty$. Then, by the inequality on displacement functions in the proof of Proposition~\ref{proposition=boundedgeneration}, for all $n$ we have that
\[
\mathrm{disp}^{S}_{\alpha_n}(\xi_n)\leq \mathrm{disp}^{(M\cup L)^N}_{\alpha_n}(\xi_n)\leq 2N(D+1).
\]

Finally, the resulting action $(\alpha,\mathcal{H}):=\lim_{\mathcal{U}}(\alpha_n,\mathcal{H}_n,\xi_n)$, and points $\xi:=[(\xi_n)_n]$ and $\eta:=[(\eta_n)_n]$ realize $D$. Here, observe that $(\alpha,\mathcal{H}) \in \mathcal{C}^{1\textrm{-}\mathrm{uniform}}$.
\end{proof}

\subsection{Closing}
\begin{proof}[Proof of Theorem~$\ref{theorem=synthesis}$]
By contradiction. Suppose that $G\geqslant M$ and $G\geqslant L$ have relative property $(\mathrm{T})$, but that $G$ fails to have property $(\mathrm{T})$. Then, by Propositions~\ref{proposition=gs} and \ref{proposition=shalom}, there must exist a realizer  $(\alpha,\mathcal{H},\xi ,\eta)$ of $D$ as in Proposition~\ref{proposition=shalom}. In particular, $\|\xi-\eta\|=D=\mathrm{dist}(\mathcal{H}^{\alpha(M)},\mathcal{H}^{\alpha(L)})$, and $\alpha$ is $1$-uniform. 

Note that $G^{\mathrm{abel}}:=G/[G,G]$  is finite. Indeed, for the abelianization map $\mathrm{ab}_G\colon G\twoheadrightarrow (G^{\mathrm{abel}},+)$, relative properties $(\mathrm{T})$ above imply that $|\mathrm{ab}_G(M)|<\infty$ and $|\mathrm{ab}_G(L)|<\infty$: otherwise, we would have non-trivial translations. (Finite generation of $G$ implies ones of $\mathrm{ab}_G(M)$ and $\mathrm{ab}_G(L)$.) Finally, observe that $G^{\mathrm{abel}}=\mathrm{ab}_G(M)+\mathrm{ab}_G(L)$.

Let $\pi$ be the linear part of $\alpha$. According to the decomposition $\mathcal{H}=\mathcal{H}^{\pi(G)}\oplus (\mathcal{H}^{\pi(G)})^{\perp}$, $\alpha$ is decomposed into $\alpha_{\mathrm{trivial}}$ and $\alpha_{\mathrm{orthogonal}}$ (this is done by decomposing the cocycle $b$ into these two summands). Because $G^{\mathrm{abel}}$ is finite, $\alpha_{\mathrm{trivial}}$ is the trivial action. We can extract $\alpha_{\mathrm{orthogonal}}$ from $\alpha$ without changing $D$ and $1$-uniformity. We, thus, may assume that $\mathcal{H}^{\pi(G)}=\{0\}$. 

Then, Proposition~\ref{proposition=si} applies; therefore, either $\xi \in \mathcal{H}^{\alpha(G)}$ or $\eta \in \mathcal{H}^{\alpha(G)}$ must hold by hypothesis $(\mathrm{GAME}^{\mathrm{inn}})$. It contradicts the assumption that $\alpha$ is $1$-uniform.
\end{proof}

\begin{remark}
Our Theorem~\ref{theorem=synthesis} is greatly generalized to Main Theorems in \cite{mimurass}. There, we deal with fixed point properties with respect to more general metric spaces (even non-linear ones); we furthermore allow some non-inner automorphisms of $G$ in type $(\mathrm{I\hspace{-.1em}I})$ moves. (In type $(\mathrm{I\hspace{-.1em}I}^{\mathrm{inn}})$ move, the automorphisms are $\mathrm{inn}(w^{-1})$ for $w\in W$.)
\end{remark}

\section*{acknowledgments}
The author is truly indebted to Michihiko Fujii for the kind invitation to the conference ``Topology and Analysis of Discrete Groups and Hyperbolic Spaces'' in June, 2016 at the RIMS, Kyoto, where this expository article had an occasion to come out. He thanks Pierre de la Harpe and Corina Ciobotaru for drawing the author's attention to the Mautner phenomenon,  Thomas Haettel and Andrei Jaikin-Zapirain for comments, 
Masahiko Kanai for providing him with the terminology ``intrinsic'', and Takayuki Okuda for discussions on moves in $(\mathrm{Game})$.  The author is supported in part by JSPS KAKENHI Grant Number JP25800033 and by  the ERC  grant 257110 ``RaWG''.

\bibliographystyle{amsalpha}
\bibliography{mimura_t_fv.bib}

\end{document}